\documentclass[12pt]{amsart}

\usepackage{amsmath}
\usepackage{amssymb}
\usepackage{ascmac}
\usepackage{amsthm}
\makeatletter

\@addtoreset{equation}{section}
\makeatother
 
\theoremstyle{plain}
\newtheorem{thm}{Theorem}[section]
\newtheorem*{thm*}{Theorem}
\newtheorem{prop}[thm]{Proposition}
\newtheorem{lem}[thm]{Lemma}

\theoremstyle{definition}
\newtheorem{ex}{Example}[section]

\theoremstyle{remark}

\newcommand{\vol}{\operatorname{vol}}

\newcommand{\diam}{\operatorname{diam}}

\newcommand{\dist}{\mathop{\mathit{d}} \nolimits}

\usepackage{latexsym}

\title[Domain monotonicity]{A note on domain monotonicity for the Neumann eigenvalues of the
Laplacian}

\author{Kei Funano}
\address{Division of Mathematics \& Research Center for Pure and Applied Mathematics, Graduate School of Information Sciences, Tohoku University, 6-3-09 Aramaki-Aza-Aoba, Aoba-ku, Sendai 980-8579, Japan}
\email{kfunano@tohoku.ac.jp}

\date{}

\begin{document}
\maketitle

\begin{abstract}Given a convex domain and a convex subdomain we prove
 a variant of domain monotonicity for the Neumann
 eigenvalues of the Laplacian. As an application of our method we also
 obtain an upper bound for Neumann eigenvalues of the Laplacian of a
 convex domain.
\end{abstract}

%%%%%%%%%%%%%%%%%%%%%%%%%%%%
%%%%%%%%%%%%%%%%%%%%%%%%%%%%
%%%%%%%%%%%%%%%%%%%%%%%%%%%%
\section{Introduction}

Let $\Omega$ be a bounded domain in $\mathbb{R}^n$ with piecewise smooth
 boundary. For the Neumann eigenvalues
 $0=\lambda_0^N(\Omega)<\lambda^N_1(\Omega)\leq \lambda^N_2(\Omega)\leq
 \cdots \leq \lambda^N_k(\Omega)\leq \cdots$ of the Laplacian on
 $\Omega$ we prove a variant of domain monotonicity: 
\begin{thm}\label{DMthm}There exists a universal constant $C>0$ such that for any two
bounded convex domains $\Omega\subseteq \Omega'$ in $\mathbb{R}^n$ with piecewise smooth boundaries
the eigenvalues of their Neumann eigenvalues satisfies
\begin{equation*}
\lambda_k^N(\Omega')\leq C n^2 \lambda_k^N(\Omega).
\end{equation*}
\end{thm}The constant in the above theorem can be chosen as $C=(92)^2$.

The following example (for $p=1$) indicates the sharpness of the above inequality with respect to the order of $n$.

\begin{ex}\label{ex:intro2}\upshape Let $p\in [\,1,2\,]$ and $B_p^n$
  be the $n$-dimensional $\ell_p$-ball centered at the origin. Suppose that
  $r_{n,p}$ is the positive number such that $\vol(r_{n,p}B_p^{n})=1$
  and set $\Omega':=r_{n,p}B_p^n$. Then $r_{n,p}\sim n^{1/p}$ and
  $\lambda_1^N(\Omega')\geq c$ for some absolute constant $c>0$
  (\cite[Section 4 (2)]{sodin}). If the segment in $\Omega'$ connecting
  the origin and $(r_{n,p},0,0,\cdots,0)$ is approximated by a convex
  domain $\Omega$ in $\Omega'$ then $\lambda_1^N(\Omega)\sim r_{n,p}^{-2}\sim
  n^{-2/p}$.
  \end{ex}
  We remark that domain monotonicity for the Dirichlet eigenvalues
  $\lambda_1^D(\Omega)\leq \lambda_2^D(\Omega)\leq \cdots \leq
  \lambda_k^D(\Omega)\leq \cdots$ is an easy consequence of the Courant
  minimax principle (\cite{C}). Actually in this case we have
  $\lambda_k^D(\Omega')\leq \lambda_k^D(\Omega)$ for any two bounded
  domains $\Omega\subseteq \Omega'$. An example of a ball and its dumbbell like subdomain shows a sort of convexity cannot be avoided in the assumption of
  the above theorem.

  In \cite{F} the author proved that $\lambda_k^N(\Omega')\leq C
  (n\log k)^2 \lambda_{k-1}^N(\Omega)$ for some universal constant $C>0$ under the same assumption of the
  above theorem. In particular this implies
  \begin{align*}
   \lambda_k^N(\Omega')\leq C
  (n\log k)^2 \lambda_{k}^N(\Omega).
   \end{align*}Theorem \ref{DMthm} removes the $\log k$ factor and thus
  improves this inequality.

  As a byproduct of the our method we obtain an upper bound estimate of Neumann
  eigenvalues of the Laplacian concerning the P\'olya conjecture. See
  Section \ref{UPNE}.
%  In the last section we consider the case of nonconvex $\Omega'$.
%%%%%%%%%%%%%%%%%%%%%%%%%%%%
%%%%%%%%%%%%%%%%%%%%%%%%%%%%
%%%%%%%%%%%%%%%%%%%%%%%%%%%%
\section{Preliminaries}\label{sec:Dirichlet eigenvalues}

Let $\Omega$ be a bounded domain in a Euclidean space with piecewise smooth
    boundary and $\{\Omega_i\}_{i=0}^{l}$ be a finite partition of $\Omega$ by subdomains;
    $\Omega=\bigcup_i \Omega_i$ and $\vol (\Omega_i\cap \Omega_j)=0$ for different $i\neq j$. The following proposition was due to Buser \cite[8.2.1 Theorem]{B}. See \cite{G} for an weaker form and also \cite[Proposition 6.1]{FS2} for
    generalization.
\begin{prop}[{\cite{B}}]\label{lower}Under the above situation, we have
\begin{align*}
    \lambda_{l+1}^N(\Omega)\geq \min_{i=0,1,\cdots,l}\lambda_1^N(\Omega_i).
\end{align*}
 \end{prop}
 We use the following relation between diameter and the first positive Neumann eigenvalue of the Laplacian under the convexity assumption. 

  \begin{prop}[{\cite[(1.2)]{PW}}]\label{emil:prop}Let $\Omega$ be a bounded convex domain in a Euclidean
    space. Then we have
    \begin{align*}
    \lambda_1^N(\Omega) \geq \frac{\pi^2 }{(\diam \Omega)^2}.
     \end{align*}
    \end{prop}

  From Proposition \ref{lower} we can
  obtain a lower bound of eigenvalues of the Laplacian once we give a
  partition. In the proof of Theorem \ref{DMthm} we use the Voronoi
  partition to get the lower bound of $\lambda_k^N(\Omega)$.
 
  Let $X$ be a metric space and $\{x_i\}_{i\in I}$ be a subset of
  $X$. For each $i\in I$ we define the \emph{Voronoi cell} $C_i$
  associated with the point $x_i$ as
  \begin{align*}
   C_i:= \{x\in X \mid \dist(x,x_i)\leq \dist(x,x_j) \text{ for all
   }j\neq i   \}.
   \end{align*}Note that if $X$ is a bounded convex domain $\Omega$ in a
   Euclidean space then $\{C_i\}_{i\in I}$ is a convex partition of
   $\Omega$ (the boundaries $\partial C_i$ may overlap each
   other). Observe also that if the balls $\{ B(x_i,r)\}_{i\in I}$ of
   radius $r$ covers $\Omega$ then $C_i \subseteq B(x_i,r)$, and thus
   $\diam C_i \leq 2r$ for any $i\in I$.

   \section{Proof of Theorem \ref{DMthm}}
   The following lemma is a key to prove Theorem \ref{DMthm}. 
\begin{lem}\label{keylemma}Let $\Omega$ be a bounded convex domain in $\mathbb{R}^n$
 with a piecewise smooth boundary. Given $r>0$ suppose that $\{
 x_i\}_{i=0}^{l}$ is an $r$-separated set in $\Omega$. Then
\begin{align*}
r\leq \frac{45 n}{\sqrt{\lambda_l^N(\Omega)}}.
\end{align*}
 \end{lem}

 To prove Lemma \ref{keylemma} we use the boundary concentration
 inequality established in \cite{FS,FS2} and a variant of it. The
 boundary concentration inequality is an analogue of the exponential concentration
 inequality due to Gromov and Milman (\cite{GM}).

 For a
   subset $A$ of a metric space $X$ and $r>0$, $B_r(A)$ denotes the
   closed $r$-neighborhood of $A$.

\begin{lem}[{Boundary concentration inequality, \cite[Proposition
 2.1]{FS}}]\label{Blem1}Let $\Omega$ be a (not necessarily convex) bounded domain in
 $\mathbb{R}^n$ with a piecewise smooth boundary and let $\mu$ be the
 uniform probability measure on $\Omega$. For any $r>0$ we have
\begin{align*}
    \mu(\Omega \setminus B_r(\partial \Omega))\leq \exp\Big(1-\sqrt{\lambda_1^D(\Omega)}r\Big).
\end{align*}
\end{lem}
 For a bounded domain $\Omega$ with a piecewise
 smooth boundary $\partial \Omega$ and a piecewise smooth open domain $U$ of $\partial \Omega$ we consider
 the following mixed eigenvalue problem:
 \begin{align*}
  \Delta \phi = -\lambda \phi  \text{ on } \Omega, \phi=0 \text{ on }
  \partial \Omega \setminus U , \text{ and }\frac{\partial
  \phi}{\partial \nu}=0 \text{ on }U,
  \end{align*}where $\nu$ is the outer unit normal. On this problem
  the eigenvalues consists
  of a discrete positive sequence
  (\cite[Theorem 1 of Chapter I]{C}). Let
  $\lambda_k^{U}(\Omega)$ be the $k$th eigenvalue of the problem.

  The proof of the following lemma is the same as the proof of
  \cite[Proposition 2.1]{FS} and we omit it.

\begin{lem}\label{Blem2}Let $\Omega$ be a (not necessarily convex) bounded domain in
 $\mathbb{R}^n$ with a piecewise smooth boundary and $U$ be a piecewise
 smooth open domain of $\partial \Omega$. Let $\mu$ denotes the
 uniform probability measure on $\Omega$. For any $r>0$ we have
\begin{align*}
    \mu(\Omega \setminus B_r(\partial \Omega\setminus U))\leq \exp\Big(1-\sqrt{\lambda_1^U(\Omega)}r\Big).
\end{align*}
\end{lem}

Since convex domains in $\mathbb{R}^n$ enjoy the CD$(0,n)$ condition in
the sense of Lott-Villani-Sturm (\cite{LV,St1,St2}) and CD$(0,n)$ spaces satisfy the
Bishop-Gromov volume comparison theorem (\cite[Theorem 2.3]{St2}) we obtain the following
lemma. Here we give a direct proof so that this paper become self-contained. 
\begin{lem}[{Bishop-Gromov inequality}]Let $\Omega$ be a convex domain
 in $\mathbb{R}^n$. Then for any $x\in \Omega$ and any $R>r>0$ we have
 \begin{align}\label{BGI}
 \frac{ \vol (B(x,r)\cap \Omega)}{\vol (B(x,R)\cap \Omega)}\geq \Big(\frac{r}{R}\Big)^n.
  \end{align}
 \begin{proof}Recall that the Brunn-Minkowski inequality (\cite{G}) states that for
  any two
  measurable subsets $A$, $B$ in $\mathbb{R}^n$ and $t\in [\,0,1 \,]$ we have
  \begin{align}\label{BMI}
   \vol((1-t)A+tB)^{\frac{1}{n}}\geq (1-t)\vol (A)^
   {\frac{1}{n}} + t\vol (B)^
   {\frac{1}{n}}, 
   \end{align}where
  \begin{align*}
   (1-t)A+tB:=\{ (1-t)x+ty \mid x\in A, y\in B\}.
   \end{align*}
  The inequality (\ref{BGI}) follows by putting $A:=\{x\}$,
  $B:=B(x,R)\cap \Omega$, and $t:=\frac{r}{R}$ in (\ref{BMI}). This
  completes the proof.
  \end{proof}
 \end{lem}
\begin{proof}[Proof of Lemma \ref{keylemma}]Let $B_i:=B(x_i,r/8)\cap \overline{\Omega}$. For any positive number $r'<\frac{1}{8}\min_{i\neq
 j}\dist(B_i,B_j)$ we set $\widetilde{B}_i:=B_{r'}(B_i)\cap \overline{\Omega}$. We also set $A:=\bigcup_{i=0}^l \widetilde{B}_i$. Then the
 (usual) domain monotonicity gives

 \begin{align*}
 \lambda_l^N(\Omega)\leq
\left\{
\begin{array}{ll}
 \lambda_{l+1}^{A \cap \partial \Omega}(A) &  \ \ (A\cap\partial \Omega
  \neq \emptyset), \\\
\lambda_{l+1}^{D}(A) &  \ \ (A\cap \partial \Omega =\emptyset) . \\
\end{array}
\right.
\end{align*}Putting $\nu_1(\widetilde{B}_i):= \lambda_1^{\widetilde{B}_i\cap
 \partial \Omega}(\widetilde{B}_i)$ if $\widetilde{B}_i\cap
 \partial \Omega\neq \emptyset$ and $\nu_1(\widetilde{B}_i):= \lambda_1^{D}(\widetilde{B}_i)$ if $\widetilde{B}_i\cap
 \partial \Omega= \emptyset$ we then get
\begin{align*}
    \lambda_l^N(\Omega)\leq \max_{i=0,1,\cdots,l} \nu_1(\widetilde{B}_i).
\end{align*}Suppose that the maximum of the right-hand side is attained
 by $\nu_1(\widetilde{B}_{i_0})$. Let $\mu_{i_0}$ be the uniform probability measure on $B_{i_0}$. 
 The Bishop-Gromov inequality implies that
\begin{align*}
    \mu_{i_0}(B_{i_0})\geq \Big(\frac{\frac{r}{8}}{\frac{r}{8}+r'}\Big)^n \geq \frac{1}{
    5^n}.
\end{align*}Thus Lemmas \ref{Blem1} and \ref{Blem2} imply that 
\begin{align*}
 \exp\Big(1-\sqrt{\nu_1(\widetilde{B}_{i_0})}s\Big)\geq
\left\{
\begin{array}{ll}
 \mu_{i_0}(\widetilde{B}_{i_0}\setminus B_s(\partial
  \widetilde{B}_{i_0}\setminus \partial \Omega)) &  \ \ (\widetilde{B}_{i_0}\cap\partial \Omega
  \neq \emptyset), \\\
 \mu_{i_0}(\widetilde{B}_{i_0}\setminus B_s(\partial \widetilde{B}_{i_0})) &  \ \ (\widetilde{B}_{i_0}\cap \partial \Omega =\emptyset) . \\
\end{array}
\right.
\end{align*}This shows that if $\exp
 (1-\sqrt{\nu_1(\widetilde{B}_{i_0})}s )<1/5^n$ then $r'\leq s$. 
That is, as long as
 \begin{align*}
  \frac{1}{\sqrt{\nu_1(\widetilde{B}_{i_0})}}(1+n\log 5)<s,
  \end{align*}we have $r'\leq s$. Therefore we get
 \begin{align*}
 r' \leq \frac{1}{\sqrt{\nu_1(\widetilde{B}_{i_0})}}(1+n\log 5)\leq \frac{4n}{\sqrt{\lambda_l^N(\Omega)}}.
  \end{align*}Since $r'$ can be sufficiently close to $\frac{1}{8}\min_{i\neq
 j}\dist(B_i,B_j)$ and $\min_{i\neq
 j}\dist(B_i,B_j)$ is at least $3r/4$ we obtain the lemma. 
 \end{proof}

\begin{proof}[Proof of Theorem \ref{DMthm}]Let
 $R:=46n/\sqrt{\lambda_k^N(\Omega')}$. We take a maximal $R$-separated net $\{x_i\}_{i=0}^l$ in
 $\Omega'$. If $l\geq k$ then by Lemma \ref{keylemma}
 we have
 \begin{align*}
  \frac{46n}{\sqrt{\lambda_k^N(\Omega')}}=R\leq
  \frac{45n}{\sqrt{\lambda_l^N(\Omega')}}\leq  \frac{45n}{\sqrt{\lambda_k^N(\Omega')}}.
  \end{align*}This is a
 contradiction. Hence $l\leq k-1$.
 
  Let $y_0, y_1,y_2,\cdots,y_l$ be maximal $R$-separated points in
   $\Omega'$, where $l\leq k-1$. By the
 maximality we have $\Omega'\subseteq
   \bigcup_{i=0}^{l} B(y_i,R)$. If $\{
   \Omega_i' \}_{i=1}^{l}$ is the Voronoi partition associated with
   $\{y_i\}$ then we have $\diam \Omega_i'\leq 2R$. Setting
   $\Omega_i:=\Omega_i'\cap \Omega$ we get
   $\Omega=\bigcup_{i=0}^l\Omega_i$ and $\diam \Omega_i\leq 2R$. Since each
   $\Omega_i$ is convex, Proposition \ref{emil:prop} gives
   $\lambda_1^N(\Omega_i)\geq \pi^2/(2R)^2$. Applying Proposition \ref{lower} to
   the covering $\{\Omega_i\}$ we obtain
   \begin{align*}
    \lambda^N_{k}(\Omega)\geq \lambda^N_l(\Omega)\geq \pi^2/\{(2 R)^2\}\geq
    \lambda^N_k(\Omega')/(92n)^2,
    \end{align*}which yields the conclusion of the theorem. This completes the proof.
\end{proof}

\section{An upper bound for Neumann eigenvalues}\label{UPNE}
Let $\Omega$ be a bounded domain with piecewise smooth boundary in a
Riemannian manifold. We denote $\lambda_k^N(\Omega)$ the $k$-th
positive Neumann eigenvalue of the Laplacian on $\Omega$, counted with
multiplicities. Applying the method of the previous section we prove the
following. Recall that $\Omega$ is \emph{convex} iff any minimizing
geodesic connecting two points in $\Omega$ is included in $\Omega$.
 \begin{thm}\label{UPP}There is a universal constant $C>0$ satisfying the following. Let $M$ be an $n$-dimensional complete Riemnnian manifold of
  nonnegative Ricci curvature and $\Omega$ be a bounded convex domain in $M$
   with piecewise smooth boundary. Then we have
   \begin{align}\label{UPB}
    \lambda_k^N(\Omega)\leq C\Big(\frac{k}{\omega_n \vol \Omega}\Big)^{\frac{2}{n}},
    \end{align}where $\omega_n$ is the volume of an $n$-dimensional
  Euclidean unit ball.
  \end{thm}
  One can take $C=10(46)^2$ in the above theorem.
  \begin{proof}Note first that Lemma \ref{keylemma} also holds in our
   nonlinear setting since the assumption of the nonnegativity of Ricci
   curvature and convexity of $\Omega$ imply the Bishop-Gromov
   inequality. Hence as in the proof of Theorem \ref{DMthm} at most $k$ balls of radius
   $46 n/\sqrt{\lambda_k^N(\Omega)}$ covers $\Omega$. Since the Bishop
   inequality gives $\vol (B(x,R)\cap \Omega) \leq w_n R^n$ for any $x\in
   \Omega$ and $R>0$ (this follows directly from (\ref{BGI}) by letting
   $r \rightarrow 0$) we have
   \begin{align*}
    \vol \Omega \leq k \omega_n \Big(\frac{46n}{\sqrt{\lambda_k^N(\Omega)}}\Big)^n.
    \end{align*}Thus $\Gamma(\frac{n}{2}+1)^{\frac{2}{n}} \sim n$ shows (\ref{UPB}). This completes the proof.
   \end{proof}
  Let us review the previous known results relating with Theorem
  \ref{UPP}. P\'olya conjectured that 
   \begin{align}\label{PolC}
    \lambda_k^N(\Omega)\leq 4\pi^2 \Big(\frac{k}{\omega_n \vol \Omega}\Big)^{\frac{2}{n}}
    \end{align}holds for any $k$ and any bounded domain $\Omega$ in
    $\mathbb{R}^n$ (\cite[Chapter XIII]{P}), that is, the principal term of the Weyl law
provides a bound for the eigenvalues. Theorem \ref{UPP} says that
    the conjecture is affirmative up to
    a multiplicative constant factor under the convexity assumption.

    The
    conjecture is affirmative in the case of $k=1,2$. The case of $k=1$
    was proved by Weinberger removing some condition supposed by Szeg\"o (\cite{W,Sz}). The case of $k=2$ was
    solved affirmatively bv Bucur and Henrot (\cite{BH}). For
    general $k$ P\'olya (\cite{P2}) proved periodic tiling domains satisfies
    (\ref{PolC}) and later Kellner (\cite{K}) removed the periodic
    condition. Recently Filonov, Levitin, Polterovich and Sher showed the P\'olya
    conjecture is affirmative for planar discs and planar sectors (\cite{FLPS}).

    Using harmonic analysis Kr\"{o}ger (\cite[Corollary 2]{Kr}) proved that
    \begin{align*}
     \lambda_k^N(\Omega)\leq (2\pi)^2\Big(\frac{n+2}{2}\Big)^{\frac{2}{n}}
     \Big(\frac{k}{\omega_n \vol \Omega}\Big)^{\frac{2}{n}}\lesssim \Big(\frac{k}{\omega_n \vol \Omega}\Big)^{\frac{2}{n}}
     \end{align*}holds for any (not necessarily convex) bounded domain $\Omega$ in $\mathbb{R}^n$
     with piecewise smooth boundary, which
     is sharper than our estimate (\ref{UPB}).

     For a domain in a Riemannian manifold, Korevaar  (\cite[(0.3) Theorem]{Ko}) obtained much more
     general and deep result. For example from his result one can see that
     \begin{align*}
      \lambda_k^N(\Omega)\leq c_n \Big(\frac{k}{\omega_n \vol \Omega}\Big)^{\frac{2}{n}}
      \end{align*}holds for a (not necessarily convex) bounded domain $\Omega$ in a manifold with
      non-negative Ricci curvature, where $c_n>0$ is a numerical
      constant depending only on $n$. He also got an estimate for a domain
      in a manifold with a lower Ricci curvature bound. The constant $c_n$ comes from an
      upper bound of a number of balls of (some fixed) radius $5^{-1}R$ that covers
      an annulus $B(x,2R)\setminus B(x,R)$. Using the Bishop-Gromov
      inequality the upper bound can be
      estimated from above by an exponential in $n$. 

      Colbois and Maerten (\cite[Theorem 1.3]{CM}) showed that for each bounded domain $\Omega$
     in a complete Riemannian manifold with Ricci curvature bounded
     below by $-(n-1)a^2$, $a\geq 0$ we have
     \begin{align*}
      \lambda_k^N(\Omega)\leq A_n a^2
      +B_n\Big(\frac{k}{\vol \Omega}\Big)^{\frac{2}{n}}
      \end{align*}for some numerical constants $A_n, B_n>0$ depending
      only on $n$. The constant $B_n$ is
      depending on the covering number $C(r)$ such that each ball of
      radius $4r$ in $M$ may be covered by $C(r)$ balls of some fixed
      radius $r$. The covering number can be estimated from above by an exponential in
      $n$ via the Bishop-Gromov inequality.

    The same proof applies for an upper bound for eigenvalues of the
    Laplacian on closed Riemannian manifolds of nonnegative Ricci
    curvature. Since there is no boundary one can give a simpler proof of
    the corresponding statement of Lemma \ref{keylemma} as follows.

    \begin{lem}Let $M$ be an $n$-dimensional closed Riemannian manifold of nonnegative
     Ricci curvature. Given $r>0$ suppose that $\{
 x_i\}_{i=0}^{l}$ is an $r$-separated set in $M$. Then
\begin{align*}
r\leq \frac{8n}{\sqrt{\lambda_l(M)}},
\end{align*}where $\lambda_l(M)$ is the $l$th nontrivial eigenvalue of
     the Laplacian on $M$, counted with mltiplicities.
     \begin{proof} Let $B_i:=B(x_i,r/2)$. %For any positive number $r'<\frac{1}{8}\min_{i\neq
 %j}\dist(B_i,B_j)$ we set $\widetilde{B}_i:=B_{r'}(B_i)$.
 % Since the maximality of $\{x_i\}_{i=0}^l$ gives
 %$\Omega \setminus B_r(\partial \Omega) \subseteq
 %\bigcup_{i=0}^lB(x_i,r)  $, we have $\Omega \subseteq
 %\bigcup_{i=0}^lB(x_i,(1+L)r)  $ which leads to
 %\begin{align*}
 % 0<8r'<\min_{i\neq j}\dist(B_i,B_j)\leq 2(1+L)r.
  %\end{align*}
     Since $B_i\cap B_j=\emptyset$ for distinct $i,j$ the (usual) domain monotonicity yields
\begin{align*}
    \lambda_l(M)\leq
 \lambda_{l+1}^{D}\Big(\bigcup_{i=0}^l{B}_i \Big)\leq \max_{i=0,1,\cdots,l} \lambda_1^{D}({B}_i).
\end{align*}By virtue of Cheng's eigenvalue comparison theorem
      (\cite[Theorem 1.1]{Ch}) we have
      \begin{align*}
       \lambda_1^D(B_i)\leq \lambda_1^D(\{x\in \mathbb{R}^n \mid
       \|x\|\leq r/2\})= \Big(\frac{r}{2}\Big)^{-2}(j_{\frac{n}{2}-1,1})^2,
       \end{align*}where $j_{\frac{n}{2}-1,1}$
     denotes the first positive zero of the Bessel function
     $J_{\frac{n}{2}-1}$ (\cite[Theorem 4 of Chapter II]{C}).
     Since $j_{\frac{n}{2}-1,1} \leq 2n$ (\cite[P 486 (5)]{watson}) we thereby
     get
      \begin{align*}
        \lambda_l(M)\leq\Big(\frac{r}{2}\Big)^{-2}(j_{\frac{n}{2}-1,1})^2\leq 16r^{-2}n^2.
       \end{align*}This completes the proof.
      \end{proof}
     \end{lem}
     The same proof of Theorem \ref{UPP} thus implies the following.
     \begin{thm}\label{UPP2}There is a universal constant $C>0$ satisfying the following. Let $M$ be an $n$-dimensional closed Riemnnian manifold of
  nonnegative Ricci curvature. Then we have
   \begin{align}\label{UPB2}
    \lambda_k(M)\leq C\Big(\frac{k}{\omega_n \vol
    M}\Big)^{\frac{2}{n}}.
   \end{align}
     \end{thm}The above constant $C$ can be taken as $C=640$. 

     Buser proved in \cite[Satz 7]{B2} (see also \cite{B3}) that
     \begin{align*}
      \lambda_k(M)\leq \frac{(n-1)^2}{4}a^2 + c_n \Big(\frac{k}{\vol M}\Big)^{\frac{2}{n}}
      \end{align*}for a closed Riemannian manifold with Ricci curvature
      bounded below by $-(n-1)a^2$, where $c_n\sim n$. Li and Yau gave a
      similar and shaper estimate for a closed manifold with a lower Ricci
      curvature bound by using a covering and comparison method. In particular they gave an
      upper bound
      \begin{align*}
       \lambda_k(M)\leq
       n(n+4)\omega_n^{\frac{4}{n}}\Big(\frac{k+1}{\omega_n\vol
       M}\Big)^{\frac{2}{n}}\lesssim  \Big(\frac{k}{\omega_n \vol M}\Big)^{\frac{2}{n}}
       \end{align*}for a closed manifold $M$ with non-negative Ricci
       curvature (\cite[Theorem 17]{LY}). Their estimate is the same with our estimate
       (\ref{UPB2}) in order. Our method looks somewhat simpler than
       their method. 
%%%%%%%%%%%%%%%%%%%%%%%%%%%%
%%%%%%%%%%%%%%%%%%%%%%%%%%%%
%%%%%%%%%%%%%%%%%%%%%%%%%%%%
\subsection*{{\rm Acknowledgements}}
The author would like to Misha Muraviev for pointing out the error of
the first version of the paper. He also thanks anonymous referees for
suggestion.
%%%%%%%%%%%%%%%%%%%%%%%%%%%%

\end{document}